\documentclass[11pt]{amsart}

\usepackage{amsmath}
\usepackage{amssymb}
\usepackage{graphicx}
\usepackage{color}

\newtheorem{theorem}{Theorem}[section]
\newtheorem{lem}[theorem]{Lemma}
\newtheorem{cor}[theorem]{Corollary}

\theoremstyle{definition}
\newtheorem{definition}[theorem]{Definition}

\theoremstyle{remark}

\theoremstyle{question}

\numberwithin{equation}{section}

\DeclareMathOperator*{\esssup}{ess\,sup}
\DeclareMathOperator*{\essinf}{ess\,inf}
\def\V{\Vert}

\keywords{Variable Lebesgue spaces, Fixed Point Property, nonexpansive mappings, Banach function lattices, Banach space geometry, order continuity. }
\begin{document}

\title[Fixed point properties and reflexivity  ]{ Fixed point properties and reflexivity in  variable Lebesgue spaces   }
\author{}

\author[T. Dom\'{\i}nguez ]{T. Dom\'{\i}nguez Benavides}
\address[T. Dom\'{\i}nguez Benavides]{Departamento de  An\'{a}lisis Matem\'{a}tico, Universidad de
Sevilla, C. Tarfia, s/n. 41012, Sevilla, Spain}
\email{\tt tomasd@us.es}

\author[M. A. Jap\'on ]{M. A. Jap\'on }
\address[M. A. Jap\'on ] {Departamento de An\'{a}lisis Matem\'{a}tico, Universidad de
Sevilla, C. Tarfia, s/n. 41012, 41012, Sevilla, Spain}
\email{\tt japon@us.es}

\thanks{The  authors  are partially supported by MICIU from the Spanish Government,
Grant PGC2018-098474-B-C21 and the Andalusian Regional Government, Grant
FQM-127. }

\begin{abstract} In this paper
the weak fixed point property ($w$-FPP) and the fixed point property (FPP) in Variable Lebesgue Spaces are studied. Given $(\Omega,\Sigma,\mu)$  a $\sigma$-finite measure and $p(\cdot)$ a variable exponent function, the $w$-FPP is completely characterized for the variable Lebesgue space $L^{p(\cdot)}(\Omega)$ in terms of the  exponent function $p(\cdot)$ and the absence of 
an isometric copy of $L_1[0,1]$. In particular,  every reflexive $L^{p(\cdot)}(\Omega)$ has the FPP and 
our results bring to light  the existence of some nonreflexive variable Lebesgue spaces satisfying the $w$-FPP, in sharp contrast with  the classic Lebesgue $L^p$-spaces. In connection with the FPP, we prove that Maurey's result for $L^1$-spaces  can be extended to the larger class of variable $L^{p(\cdot)}(\Omega)$ spaces with order continuous norm, that is,  every 
  reflexive subspace of $L^{p(\cdot)}(\Omega)$ has the FPP. 
Never\-theless,  Maurey's converse does not longer hold in the variable setting, since some   nonreflexive  subspaces of  $L^{p(\cdot)}(\Omega)$   satisfying the FPP can be found.  As a consequence, we discover that several nonreflexive Nakano sequence spaces $\ell^{p_n}$  do have the FPP endowed with the  Luxemburg norm. 
As far as the authors are concerned,  this family of  sequence  spaces gives rise to  the first  known nonreflexive    classic Banach spaces  enjoying  the FPP   without requiring of any renorming procedure. The failure of asympto\-tically isometric copies of $\ell_1$ in $L^{p(\cdot)}(\Omega)$ is also analyzed.

\end{abstract}

\maketitle

\section{Introduction }

The class  of Variable Lebesgue Spaces (VLSs) arises as a  generalization of
 classic Lebes\-gue spaces $L^p(\Omega)$,  when  the constant exponent $p$ is replaced with a variable
exponent  function $p(\cdot)$. The resulting  function space $L^{p(\cdot)}(\Omega)$ shares many of the geometrical  properties of the  $L^p$-spaces but also differ from them in some sort of  interesting and unexpected forms (as an example, VLSs are not translation invariant unless the exponent function $p(\cdot)$ is constant). Variable Lebesgue spaces  can  be traced back   in the literature  to 1931 \cite{Orl}  and they    lie within the scope of  the more general class  of   modular function  spaces, initially  defined  by H. Nakano \cite{N,Na,Nak} and studied  by Orlicz and Musielak \cite{MO}. However, the last two decades 
 have witnessed an explosive development  in the analysis of the intrinsic structure of  VLSs by their own right,   in particular,  since  M. R\.{u}\v{z}i\v{c}ka  discovered that they constitute a natural functional setting for the mathematical
model of electrorheological fluids \cite{Ru}. As Banach function spaces,
the  structure and geometrical properties of VLSs connected to Harmonic Analysis and some other areas within the scope of Functional Analysis  have been studied in  \cite{cruz, otro, KR, LuPiPo} and the references therein. 
However, a precise Fixed Point Theory for nonexpansive operators on this family of Banach  spaces seems to be in a very incipient state.
\medskip

We recall that a Banach space $(X,\V\cdot\V)$ is said to have the fixed point property (FPP) if every nonexpansive operator $T:C\to C$, with $C$ a closed convex bounded subset of $X$, has a fixed point. Besides, $X$ is said to have the weakly fixed point property  ($w$-FPP) when the above holds for all domains which are convex  and weakly compact. Here nonexpansiveness means  $\V Tx-Ty\V\le \V x-y\V$ for every $x,y\in C$  (note that nonexpansiveness is  strictly subject to the underlying norm). The FPP and $w$-FPP are equivalent under reflexivity and these properties have been extensively studied  in the framework of Banach spaces for the last 60 years, during which 
multiple and robust  connections   were displayed linking this area of Me\-tric Fixed Point Theory with  Geometry and Renorming Theory in Banach spaces. 
Although the first positive results for the existence of fixed points for nonexpansive mappings date back to 1965 \cite{Ki},   this theory is very far from being complete and there are still many interesting open problems left and some unsolved long-standing conjectures. In fact, 
 it was long believed that all Banach spaces fulfilling the FPP had to be reflexive. 
In 2008 P.K. Lin proved
 that this  statement was untrue by  providing the sequence space $\ell_1$ with an equivalent norm that let it have the FPP \cite{Lin-2008}. This  automatically meant that the fixed  point property could be extended  beyond reflexivity  (see also 
 \cite{ja, jfa, DJLT, DLT,  Maria-Carlos-2010, FG, Lin-2010}). In fact, as a consequence of some results included in the paper, we will show  that this is the case in some particular classes of Variable Lebesgue Spaces,   where neither reflexivity nor any renorming argument are needed for establishing the FPP.

\medskip

The organization of the article is the following: 

\medskip

In the second section we develop some preliminary results concerning modular function spaces and mainly related   to the class of  variable Lebesgue spaces, the reflexivity condition and the order continuity of the Luxemburg norm, which is the standard norm for VLSs.  
 
 \medskip

 The third section focuses on the proper study of   the $w$-FPP.  A complete description of  those VLSs satisfying the $w$-FPP is obtained in terms of the exponent function $p(\cdot)$. 
Due to Alspach's example \cite{Al}, who proved  that $L^1[0,1]$ fails to have the $w$-FPP,  we know that
classic  Lebesgue spaces $L^p(\Omega)$  have the $w$-FPP if and only if  they are reflexive. We can assert that this equivalence does not longer hold in our variable  setting. Furthermore,  we establish  the  equivalence between the $w$-FPP and the absence of an isometric copy of $L^1[0,1]$ when restricted to the family of VLSs.

\medskip

One significant breakthrough in Metric Fixed Point Theory  supporting  the  still open question as to  whether  \lq\lq   FPP is implied by reflexivity"  is  due to B. Maurey  in  1980.
Despite the fact that the Lebesgue space $L^1[0,1]$ fails to have the $w$-FPP \cite{Al}, B. Maurey proved that every closed reflexive subspace of $L^1[0,1]$ does have the FPP \cite{maurey}. Almost two decades later,  P. Dowling and C. Lennard  \cite{DL} obtained  a converse statement: every  closed subspace of $L^1[0,1]$ with the FPP is reflexive. Hence,  the FPP  is completely characterized by reflexivity within the family of closed subspaces of $L^1[0,1]$.

In the fourth section we prove that  Maurey's result can   be  literally extended to (nonreflexive) VLSs with order continuous Luxemburg norm, that is, {\it every  closed reflexive subspace of an order continuous VLS   satisfies the FPP}. Surprisingly,  P. Dowling and C. Lennard's converse does not hold in the variable framework, that is,   under some slightly weak assumptions, we can   prove that  {\it for every nonreflexive VLS there exists a further closed nonreflexive subspace  that does verify the FPP}.  Near-infinity concentrated norms defined   in \cite{jfa} will become an essential tool in our drive and will lead us to discover that 
 Nakano spaces $\ell^{p_n}$ when $(p_n)\subset (1,+\infty)$ and $\lim_n p_n=1$ endowed with the Luxemburg norm are nonreflexive Banach spaces that do have the FPP  (see Corollary \ref{fpp2} and comments afterwards). Note that, up to this stage, no classic nonreflexive Banach space endowed with its standard norm was known to have the FPP. 
 
 \medskip

Finally, it is well-known that the failure of the FPP can be linked with the existence of an asymptotically isometric copy of $\ell_1$ (see \cite{DL,handbook} and references therein). By using  the subsequence splitting lemma for Banach lattices \cite{W}, in the last section of the paper
 we analyze the failure of asymptotically isometric copies of $\ell_1$ in VLSs, unless the trivial case is considered (when $\ell_1$ is already contained isometrically in $L^{p(\cdot)}(\Omega)$).  Some open questions and new insights sparked by the previous results are also displayed. 
We would like to remark  that the study of the  fixed point properties  in VLSs  enables us to highlight,  yet again,    that the variable counterpart of the $L^p$-spaces  exhibits a much richer and heterogeneous structure giving rise to a variety of new problems and lines of research that do not occur for the classic Lebesgue spaces.

\section{Preliminaries and reflexivity in Variable Lebesgue spaces}

\begin{definition}\label{def1}
Let $ \mathcal{X} $ be an arbitrary vector space.
\begin{itemize}
\item[(a)]
A functional $ \rho:\mathcal{X}\rightarrow [0, \infty] $ is called a convex modular if for  $ x, y\in \mathcal{X} $:\\
$ (i)$ $\rho(x)=0$ if and only if $x=0$;\\
$(ii)$ $ \rho(\alpha x)=\rho(x) $ for every scalar $ \alpha $ with $ |\alpha|=1 $;\\
$(iii)$  $ \rho(\alpha x+\beta y)\leq \alpha\rho(x)+\beta\rho(y) $ if $ \alpha+\beta=1 $ and $ \alpha, \beta \geq 0 $.
\item[(b)]
A modular $ \rho $ defines a corresponding modular space, i.e. the vector space $ \mathcal{X}_\rho $ given by
$\{x\in \mathcal{X}: \rho(x/\lambda)<\infty  \text{ for some  } \lambda >0\}. $

\end{itemize}

\end{definition}
Given a vector space $\mathcal{X}$ with a convex modular $\rho$,
the formula
\begin{align*}
\|x\| = \inf\left\{\alpha>0: \rho\left(\frac{x}{\alpha}\right)\leq 1\right\} \mbox{ for } x\in \mathcal{X}_\rho,
\end{align*}
defines a norm which is frequently called the Luxemburg norm and $\mathcal{X}_\rho$ endowed with this norm is a Banach space.

\medskip

Throughout  this paper,
 $(\Omega,\Sigma,\mu)$ will be  a $\sigma$-finite  measure space. We will always assume that the measure is complete. Let  $p:\Omega\to [1,+\infty]$ be a measurable function and we consider the vector space $\mathcal{X}$ of all measurable functions $g:\Omega\to \mathbb{R}$.
Define the modular
\begin{equation}\label{mod}
\rho(g):=\int_{\Omega_{f}}|g(t)|^{p(t)} d\mu+ \displaystyle{\esssup_{ p^{-1}(\{+\infty\})}}|g(t)|,
\end{equation}
 where $\Omega_f:=\{t\in \Omega: p(t)<+\infty\}$. Alongside with $\Omega_f$ and $p^{-1}(\{+\infty\})$, we will  also distinguish the sets $p^{-1}(\{1\})$ and $\Omega^*=\Omega\setminus p^{-1}(\{1,+\infty\})$. 

 \medskip

The Variable Lebesgue Space (VLS)  $L^{p(\cdot)}(\Omega)$ is defined as the modular   space  endowed with the Luxemburg norm associated to the modular $\rho$ defined above.
It is well-known that $L^{p(\cdot)}(\Omega)$  is a  Banach function lattice whose    geometry  is strongly attached to the behaviour of the  exponent  function $p(\cdot)$. Note that  Lebesgue spaces  $L^p(\Omega)$ endowed with the standard $\V \cdot\V_p$ norm ($1\le p\le +\infty$) are particular examples of this construction just by considering the constant function $p(t)=p$ for all $t\in \Omega$.

\medskip

Following the usual notation,
 given a measurable set $E\subset \Omega$, we define
$$
 p_-(E):=\essinf_{t\in E } p(t), \quad p_+(E):=\esssup_{t\in E} p(t).
 $$
If $E=\Omega$ we just denote $p_-:=p_-(\Omega)$ and $p_+:=p_+(\Omega)$.

\medskip

 A modular space $\mathcal{X}_\rho$ is said to satisfy the $\Delta_2$-condition if there exists $M>0$ such that $\rho(2f)\leq M\rho(f)$ for every $f\in \mathcal{X}_\rho$. It is easy to prove that $L^{p(\cdot)}(\Omega)$ satisfies the $\Delta_2$-condition if $p_+(\Omega_f)<\infty$ (see \cite[Proposition 2.14] {cruz}). Moreover, in this case $\rho(g)<+\infty$ for every $g\in L^{p(\cdot)}(\Omega)$.

\medskip

A Banach lattice $X$ is said to have an order continuous norm if every monotone order bounded sequence  is convergent.  Using Lebesgue's Dominated Convergence  Theorem, it is not difficult to prove that the Luxemburg norm of a VLS is order continuous
whenever   $p_+(\Omega_f)<+\infty$ and $p^{-1}(\{\infty\})$ is the union of at most a null set and finitely many atoms. The following properties relating the modular and the Luxemburg norm in VLSs  will be used through this paper. 

\begin{lem}\label{properties} Let $(\Omega,\Sigma,\mu)$ be a $\sigma$-finite measure, $p:\Omega\to [1,+\infty]$ be an exponent function. 
\begin{itemize}

\item[$a)$]  If  $g\in L^{p(\cdot)}(\Omega)$, $g\ne 0$, then  $\rho\left(g\over \V g\V\right)\le 1$. Additionally,  $\V g\V\le \rho(g)$ when $\Vert g\Vert \ge 1$.

\item[$b)$]  Assume that $p_+(\Omega_f)<\infty$ and $g\in L^{p(\cdot)}(\Omega)$.  Then:

\begin{itemize}

\item[$b.i)$] If $a\ge 1$, $a\rho(g)\le \rho(ag)\le a^{p_+(\Omega_f)}\rho(g)$.

\item[$b.ii)$] If $0<a< 1$, $a^{p_+(\Omega_f)}\rho(f)\le \rho(af)\le a\rho(f)$.

\end{itemize}

\item[$c)$] Assume $p_+(\Omega_f)<\infty$ and $(g_n)$ is a sequence in $L^{p(\cdot)}(\Omega)$. Then:

\begin{itemize}

\item[$c.i)$]  $\lim_n \V g_n\V=1$ if and only if $\lim_n \rho(g_n)=1$.

\item[$c.ii)$]  $\lim_n \V g_n\V=0$ if and only if $\lim_n \rho(g_n)=0$.

\end{itemize}

\end{itemize}

\end{lem}

\begin{proof}
Assertion $a)$ is proved in \cite[Proposition 2.21]{cruz}. Assertions $b.i)$  and $b.ii)$  are consequences of the convexity and the definition  of the modular and assertions $c.i)$ and $c.ii)$ can be easily proved using the inequalities in $a)$ and $b)$.

\end{proof}

When the measure space $(\Omega,\sigma,\mu)$ is purely atomic, the exponent function $p(\cdot)$ can be considered  as a sequence $(p_n)_n\subset [1,+\infty]$. The corresponding VLS is denoted by $\ell^{p_n}$ and they are  usually known in the literature as  a particular class of Musielaz-Orlicz sequence spaces or simply as  Nakano spaces  \cite{Na} (note that the modular considered in \cite{Na} is $m(x)=\sum_{n=1}^\infty {1\over p_n}|x_n|^{p_n} $ for $x=(x_n)$ when $(p_n)\subset[1,+\infty)$, while in the variable exponent setting   the modular is defined by $\rho(x)=\sum_{n=1}^\infty |x_n|^{p_n} $.
It  can easily be checked that both modular spaces
  give rise to isomorphic Banach spaces when $(p_n)$ is bounded).

\medskip

 When the norm fails to be order continuous, it is a general fact  in the theory of Banach function lattices the existence of an isomorphic copy of $\ell_\infty$
  \cite[Corollary 2.4.3]{MN} (see also \cite[Proposition 4.2]{LuPiPo}). In particular every nonreflexive  function lattice contains an isomorphic copy of $\ell_\infty$ \cite[Theorem 2.4.2]{MN}.  In the specific  case of the family of Musielak-Orlicz spaces, it was proved in \cite{He} (for nonatomic $\sigma$-finite measures) and in  \cite{Ka1} (for purely atomic measures) that, in absence of the $\Delta_2$-condition, there is an isometric copy of $\ell_\infty$ when the Luxemburg norm is considered. This stronger result is essential when is to be applied to the analysis of the fixed point property, since
having an isomorphic copy of $\ell_\infty$ does not exempt  a Banach space from satisfying the $w$-FPP (see \cite{tomas}).
Although  variable Lebesgue spaces lay  within the scope of the Musielak-Orlicz class, on the sake of completeness,
 we next include a proof of when such an isometric copy of $\ell_\infty$ can be found in the proper context of this article and including measures that may have both atomic and nonatomic parts.

\begin{theorem} \label{iso} Let $ (\Omega,\Sigma,\mu) $ be a $\sigma$-finite    measure space and $p:\Omega\to [1,+\infty)$ be a measurable function. If $p_+=+\infty$,
the Banach space $L^{p(\cdot)}(\Omega)$ contains an isometric copy of $\ell_\infty$. Consequently, under these assumptions, $L^{p(\cdot)}(\Omega)$ contains an isometric copy of every separable Banach space.
\end{theorem}

\begin{proof}

Since $p_+=+\infty$, either there exists a   sequence of atoms $\{m_n\}$ such that $p(m_n)\to+ \infty$ or there exists $M$ such that the set $p^{-1} ((M,+\infty))$ does not contain any atom. In any case,  we can find a real sequence $\{p_n\}\uparrow +\infty$ such that $\mu(p^{-1}([p_n,p_{n+1})))>0$ and $(1+\frac{1}{n})^{p_n}>2^n$. Let $S_n\subset p^{-1}([p_n,p_{n+1}))$ such that $0<\mu(S_n)<+\infty$.  Hence $S_n\cap S_m=\emptyset$ if $n\ne m$.  Denote by $\{r_n\}$ the increasing sequence formed by all prime numbers greater than 1. Note that if $t\in S_{r_n^j}$ then $p(t)\ge p_{r_n^j}\ge p_j$ for all $n,j\in\mathbb{N}$. For every $n\in\mathbb{N}$ we define the function
$$
f_n(t):=\displaystyle{\sum_{j=1}^\infty x_{n,j}^{1/p(t)} \chi_{S_{r_n^j}}(t)    },\ \  \mbox{where }\  x_{n,j}= \frac{1}{2^{n+1+j}\mu(S_{r_n^j})}\ \forall j\in\mathbb{N}.
$$
By construction $\rho(f_n)= {1\over 2^{n+1}}$ which implies that
 $\V f_n\V\le 1$ for all $n\in\mathbb{N}$.  Let $\lambda>1$ and
choose $j_0$ such that $1+\frac{1}{j}<\lambda$ for $j\geq j_0$.  We have: 

\begin{eqnarray*}
 \rho(\lambda f_n)&=& \sum_{j=1}^\infty \int_{S_{r_n^j}} x_{n,j} \lambda^{p(t)} d\mu   \ge  \sum_{j=j_0}^\infty \int_{S_{r_n^j}} x_{n,j}\left(1+\frac{1}{j}\right)^{p_j} d\mu  \\&\geq&  \sum_{j=j_0}^\infty  x_{n,j}2^{j}\mu(S_{r_n^j})= \sum_{j=j_0}^\infty 2^{j} \frac{1}{2^{n+1+j}}=+\infty.\end{eqnarray*}

The previous arguments prove that $\V f_n\V=1$ for every $n\in\mathbb{N}$. Likewise, it can be checked that $\V \sum_{n=1}^\infty f_n\V=1$. At this stage, it is not difficult to conclude that the sequence $(f_n)$ spans  an isometric copy of $\ell_\infty$ in $L^{p(\cdot)}(\Omega)$ (see for instance \cite[Theorem 1]{H}).
The statement of the theorem is complete due to the fact that  $\ell_\infty$ contains an isometric copy of every separable Banach space \cite[Theorem 2.5.7]{A}.
\end{proof}

\medskip

A complete analysis of the reflexivity condition  for variable Lebesgue spaces 
was  studied  for   nonatomic  measures in \cite{LuPiPo} and  for   purely atomic measures in \cite{sun}.  In fact, for the nonatomic case  the following characterization was obtained:

\begin{theorem}  \label{LuPiPo} \cite[Theorem 3.3]{LuPiPo} Let $ (\Omega,\Sigma,\mu) $ be a $ \sigma $-finite  nonatomic measure space. The following conditions are all equivalent:
\item[$(a)$] $1<p_-\leq p_+<\infty.$
\item[$(b)$]  $L^{p(\cdot)}(\Omega)$ is uniformly convex.
\item[$(c)$] $L^{p(\cdot)}(\Omega)$ is reflexive.

 \end{theorem}

The nonatomic assumption in Theorem \ref{LuPiPo} is used by the authors exclusively in the proof of \lq\lq $(c)$ implies $(a)$". The proof of \lq\lq $(a)$ implies $(b)$" holds for every $\sigma$-finite measure space.
Actually, Theorem \ref{LuPiPo}  does not  entirely hold
when the measure contains atoms,  since  reflexivity can be obtained in absence of uniform convexity:  consider the purely atomic case $\ell^{p_n}$ for
  $p_1=p_2=1$ and $p_n=2$ for $n>2$. The VLS space $\ell^{p_n}$ fails to be uniformly convex, since it contains $\ell_1(2)$ isometrically, but it is reflexive since it is isomorphic to $\ell_2$. As reflexivity will be at the core of many of our next results,   we first aim  to achieve  a  complete characterization of reflexivity  for VLSs    including all $\sigma$-finite measures:

\begin{theorem} \label{refle} Let $ (\Omega,\Sigma,\mu) $ be an arbitrary $\sigma$-finite   measure space and let $p:\Omega\to [1,+\infty]$ be a measurable function. The following conditions are equivalent:

\begin{itemize}

\item[$i)$] $L^{p(\cdot)}(\Omega)$ is reflexive.

\item[$ii)$] $L^{p(\cdot)}(\Omega)$ contains no isomorphic copy of $\ell_1$.

\item[$iii)$]  Let $\Omega^*:=\Omega\setminus p^{-1}(\{1,+\infty\})$. Then $1<p_-(\Omega^*)\le p_+(\Omega^*)<+\infty$ and $p^{-1}(\{1,+\infty\})$ is essentially formed  by finitely many atoms at most.

\end{itemize}
\end{theorem}

\begin{proof}
$i)$ implies $ii)$ is straightforward.  Let us prove $ii)$ implies $iii)$:
 If $p^{-1}(\{1,+\infty\})$ contains infinitely many atoms, either $\ell_1$ or $\ell_\infty$ would be isometrically embedded in $L^{p(\cdot)}(\Omega)$. In any case,  $L^{p(\cdot)}(\Omega)$  would contain an isometric copy of $\ell_1$.  If $p^{-1}(\{1,+\infty\})$ contains a nonatomic set with positive measure we would arrive at the same conclusion, since either $L^1[0,1]$ or $L^\infty([0,1])$ would be isometrically embedded into $L^{p(\cdot)}(\Omega)$. If  $p_+(\Omega^*)=+\infty$,   we would obtain  an isometric copy of $\ell_1$ in $L^{p(\cdot)}(\Omega^*)$ in view of  Theorem \ref{iso} and, obviously, in $L^{p(\cdot)}(\Omega)$.  Finally, if $p_-(\Omega^*)=1$, we will later prove in Theorem \ref{nic} that it is possible to find a subspace within $L^{p(\cdot)}(\Omega)$   which is hereditarely    $\ell_1$ (and therefore, it contains $\ell_1$ isomorphically).

 Let us prove $iii)$ implies $i)$: We split  $\Omega=\Omega_a\cup \Omega_b$;  $\Omega_a$, $\Omega_b$  being the purely  atomic and the  nonatomic part of $\Omega$ respectively. From  the assumptions,  we have that $1<p_-(\Omega_b)\le p_+(\Omega_b)<+\infty$  and from Theorem \ref{LuPiPo},  we know that the variable Lebesgue space $L^{p(\cdot)}(\Omega^*)$ is uniformly convex. From $iii)$ we also know that there are some  integers $0\le r_1\le r_2$ such that   $\Omega_a\cap p^{-1}(\{1\})=\{t_1,\cdots,t_{r_1}\}$ and $\Omega_a\cap p^{-1}(\{+\infty\})=\{t_{r_1+1},\cdots, t_{r_2}\}$. Set $r:=r_2-r_1\ge 0$. Thus we can write
 $$
 \rho(g)= \int_{\Omega_b} \vert g(t)\vert^{p(t)}dt+ \sum_{i=1}^{r_1}|g(t_i)|+\sup_{r_1< i\le r_2}|g(t_i)|\quad \forall g\in L^{p(\cdot)}(\Omega).
 $$

  We aim to prove that $L^{p(\cdot)}(\Omega)$ can be renormed to be uniformly convex and therefore it is reflexive. In order to do that,
 we define the measurable function $\tilde{p}: \Omega\to (1,+\infty)$ given by $\tilde{p}(t)=2$ if $t\in \Omega_a\cap p^{-1}(\{1,+\infty\})$  and $\tilde{p}(t)=p(t)$ otherwise. We denote by $\V \cdot\V_{\tilde p}$ the Luxemburg norm rising from  the modular
  $$
 \tilde \rho (g)=
\int_\Omega \vert g(t)\vert ^{\tilde{p}(t)}dt=\int_{\Omega_b} \vert g(t)\vert^{p(t)}dt +\sum_{i=1}^{r_2}|g(t_i)|^2.
$$

  We next check that  ${1\over r_1+2}\V g\V\le \V g\V_{\tilde{p}}\le \max\{1,r\} \V g\V$ for all   $g\in L^{p(\cdot)}(\Omega)$:
Assume that $\V g\V=1$. From Lemma \ref{properties}.a) we know that $\rho(g)\le 1$, which in particular implies  that $|g(t_i)|\le 1$ for all $1\le i\le r_2$.
 This gives  $\tilde\rho(g)\le  \max\{1,r\} \rho(g)\le \max\{1,r\} $ and  $\tilde\rho\left({g\over \max\{1,r\} }\right)\le {1\over \max\{1,r\} }\tilde\rho(g)\le 1$ yielding to  $\V g\V_{\tilde{p}}\le \max\{1,r\} \V g\V$ for all $g\in L^{p(\cdot)}(\Omega)$. Assume now that $\V g\V_{\tilde p}=1$ and therefore $\tilde\rho(g)\le 1$ and $|g(t_i)|\le 1$ for all $1\le i\le r_2$. Hence, $\rho(g)\le \tilde\rho(g)+ r_1+1\le r_1+2$. By convexity, we have  $\rho\left({g\over r_1+2}\right)\le {1\over r_1+2}\rho(g)\le 1$ and $\V g\V\le r_1+2$.

Hence, we have obtained that $L^{p(\cdot)}(\Omega)$ is isomorphic to $L^{\tilde p(\cdot)}(\Omega)$ which is in turn  uniformly convex, since $1<\tilde p_-\le \tilde p_+<+\infty$  (see remark after Theorem \ref{LuPiPo}) and this concludes the proof.

\end{proof}

Finally, we would like to recall  that some fixed point results   have already appeared  for VLSs when they are considered modular spaces and focusing the notion of nonexpansivity with respect to the modular $\rho(\cdot)$ defined by (\ref{mod}) \cite{BBK, BKB, BMB}. In the next sections our goal is completely different and  addresses toward the analysis of the fixed point property when nonexpansiveness is measured  with respect to the Luxemburg norm and the potential connections  linking the  geometry and  reflexivity of the underlying variable space.

\section{Weak Fixed Point Property in Variable Lebesgue Spaces}

In this section we will obtain a characterization of  the $w$-FPP in variable Lebesgue spaces in terms of the variable exponent function $p(\cdot)$ and whether or not  $L_1[0,1]$ can be isometrically embedded in  $L^{p(\cdot)}(\Omega)$. In particular, 
we will exhibit  that there are some VLSs with the $w$-FPP which are not reflexive,  in sharp contrast to the classic  $L^p$-spaces, where $L^p(\Omega)$ has the $w$-FPP if and only if $L^p(\Omega)$ is reflexive.

\medskip

We will start with two
 technical lemmas, the first of which is just a measure theory result likely well-known. We include the proof for the sake of completeness.

\begin{lem}\label{primero} Let $(\Omega,\Sigma,\mu)$ be a $\sigma$-finite measure space and $(f_n)$ be a bounded sequence in $L_1(\mu)$.Then the following statement holds: For almost every $t\in \Omega$, the scalar sequence $\{f_n(t)\}_n$ has at least one finite accumulation point. 
\end{lem}

\begin{proof}
Since the measure is $\sigma$-finite we can assume that $\Omega=\displaystyle{\cup_{s=1}^\infty\Omega_s}$ with $\mu(\Omega_s)<\infty$. If the previous statement is proved for finite measures it follows for $\sigma$-finite measures.  Hence, without loss of generality, we can assume that the measure is finite. 
 Let $I:=\{t\in \Omega: \lim_n |f_n(t)|=+\infty\}$. It is clear the scalar sequence  $\{f_n(t)\}$ has an accumulation point if and only if $t\in \Omega\setminus I$. We will prove that $\mu(I)=0$:
 
Fix a real number $a>0$.  For $n\in\mathbb{N}$, set $A_n(a)=\{t\in I:  |f_n(t)|<a\}\subset I$. Since the sequence  $\{\chi_{A_n(a)}\}_n$ converges to zero pointwise  and
 $ \chi_{A_n(a)}(t)\le  a\chi_\Omega(t)$ for all $t\in\Omega$, using Lebesgue's Dominated Convergence Theorem, we have   $\lim_n \mu(A_n(a))=0$.
For every $k\in\mathbb{N}$,  choosing $a=2^k$ and repeating the process sucessively,  we can find a subsequence $(n_k)$ such that
$$
\mu(\{t\in I:  |f_{n_k}(t)|< 2^k\})\le {1\over 2^k}\quad \mbox{ for all $k\in\mathbb{N}$}.
$$
Take $M=\sup_n\V f_n\V_1$. By  Chebychev inequality
$$
\mu(\{t\in I: |f_{n_k}(t)|\ge 2^k\})\le \mu(\{t\in \Omega: |f_{n_k}(t)|\ge 2^k\})\le {M\over 2^k}.
$$
Hence,  for every $k\in\mathbb{N}$, we have $I=\{t\in I:  |f_{n_k}(t)|< 2^k\}\cup\{t\in I: |f_{n_k}(t)|\ge 2^k\}$, which implies that $\mu(I)\le {M+1\over 2^k}$. Taking limits when $k$ goes to infinity we finally deduce that $\mu(I)=0$. 

\end{proof}

We recall that a sequence $(x_n)\subset L^{p(\cdot)}(\Omega)$ is said to be $\rho$-bounded if $\sup_n\rho(x_n)<+\infty$ where $\rho(\cdot)$ is the modular defined by (\ref{mod}). 

 \begin{lem} \label{strictlyconvex} Let $ (\Omega,\Sigma,\mu) $ be a $ \sigma $-finite measure space and assume  that the exponent function $p(\cdot)$  verifies $1<p(t)<\infty$ a.e.  Let $u,v\in L^{p(\cdot)}(\Omega)$. Assume that  there exists a  $\rho$-bounded sequence $(x_n)$ in  $L^{p(\cdot)}(\Omega)$ verifying

{\small
\begin{equation}\label{important}
\lim_n\int_\Omega \left(\vert x_n(t)-u(t)\vert ^{p(t)}+\vert x_n(t)-v(t)\vert ^{p(t)}-2\left\vert x_n(t)-\frac{u(t)+v(t)}{2}\right\vert ^{p(t)}\right)d\mu=0.
\end{equation}
}
then $u=v$ a.e.
   \end{lem}

\begin{proof} Let  $u,v\in L^{p(\cdot)}(\Omega)$  and assume that there exists a $\rho$-bounded sequence $(x_n)$  in  $L^{p(\cdot)}(\Omega)$ verifying that the limit in  (\ref{important}) is null. Under these conditions,   we are going to find a subset $B\subset \Omega$ with $\mu(B)=0$ such that $u(t)=v(t)$ for all $t\in\Omega\setminus B$.

Note that the $\rho$-boundedness of the sequence $(x_n)$ implies that the sequence $(h_n)$ defined by  $h_n(t)=|x_n(t)|^{p(t)}$ is a bounded sequence in $L_1(\Omega)$. Using Lemma \ref{primero}, we can assume that  for almost every $t\in \Omega$, the scalar sequence $\{|x_n(t)|^{p(t)}\}_n$ has a finite accumulation point, and so does the scalar sequence $\{x_n(t)\}_n$ for almost every $t\in \Omega$.

For all $n\in\mathbb{N}$ define the function 
{\small
\begin{equation}\label{limit}
g_n(t):=\vert x_n(t)-u(t)\vert ^{p(t)}+\vert x_n(t)-v(t)\vert ^{p(t)}-2\left\vert x_n(t)-\frac{u(t)+v(t)}{2}\right\vert ^{p(t)}. 
\end{equation}}
Note that $g_n\ge 0$  by convexity and, by assumption,  $\lim_n \int_\Omega g_n(t)d\mu=0$. Extracting a subsequence,  denoted again by $(g_n)$, we can assume that $
\lim_n g_n(t)=0$  for almost every $t\in \Omega$. Thus, we can assume that there exists some  $B\subset \Omega$ with $\mu(B)=0$ such that for all $t\in \Omega\setminus B$ we have that $\lim_n g_n(t)=0$ and  there exists a subsequence $(n^t_k)$ (depending on $t$) such that $\lim_k x_{n^t_k}(t)=\alpha_t$, where $\alpha_t$ is a finite scalar.
Hence, for every $t\in \Omega\setminus B$, taking limit when $k$ goes to infinite  over the subsequence $(n^t_k)$  in (\ref{limit}) we obtain 
\begin{equation}\label{final}
|\alpha_t-u(t)|^{p(t)}+|\alpha_t-v(t)|^{p(t)}-2\left| \alpha_t-{u(t)+v(t)\over 2}\right|^{p(t)}=0.
\end{equation}
We have concluded that for all $t\in\Omega\setminus B$, the equation (\ref{final}) holds. 
The strict convexity of the function $s\to s^{p(t)}$ (since $p(t)>1$) implies that $u(t)=v(t)$ for  $t\in \Omega\setminus B$ as we wanted to prove.

\end{proof}

We recall that a Banach space $X$ is said to have weak normal structure ($w$-NS)   if for every convex weakly compact  subset  $C$ with ${\rm diam}\,(C)>0$, there exists some $x_0\in C$ such that  $\sup\{\V x_0-y\V: y\in C\}
<{\rm diam}\,(C)$.
The notion of  normal structure was initially defined by Brodskii and Milman in 1948 \cite{BM} and  W. Kirk   established  its relationship with the existence of fixed points for nonexpansive mappings:  {\it Every Banach space with $w$-NS satisfies the $w$-FPP}  \cite{Ki}.

\medskip

The main theorem of this section is the following:

\begin{theorem} \label{char} Let $ (\Omega,\Sigma,\mu) $ be a  $ \sigma $-finite    measure space and $p:\Omega\to [1,+\infty]$ be a measurable function. As usual, let us denote by $\Omega_f=\{t\in \Omega: p(t)<+\infty\}$. The following conditions are all equivalent:
\begin{itemize}

\item[1)] $L^{p(\cdot)}(\Omega)$ satisfies the weak  normal structure.

\item[2)] $L^{p(\cdot)}(\Omega)$ satisfies the w-FPP.

\item[3)]  $L^{p(\cdot)}(\Omega)$ does  not contain isometrically  $L^1[0,1]$.

\item[4)]   $p_+(\Omega_f)<+\infty$, $p^{-1}(\{+\infty\})$ contains  finitely many atoms at most and every measurable atomless  subset of $p^{-1}(\{1,+\infty\})$ is negligible.

\end{itemize}\end{theorem}

\begin{proof}

We already know that $1)\Rightarrow 2)$ and $2)\Rightarrow 3)$  from \cite{Ki} and \cite{Al}.   Clearly  $3)\Rightarrow 4)$. Indeed, if $p_+(\Omega_f)=+\infty$ we can find an isometric copy of $L^1[0,1]$ from Theorem \ref{iso}.
If $p^{-1}(\{+\infty\})$ contains infinitely many atoms, we have an isometric copy of $\ell_\infty$ and hence an isometric copy of $L^1[0,1]$. Finally, if
there is a measurable atomless    subset  contained in $p^{-1}(\{1,\infty\})$ with positive measure, once more $L^{p(\cdot)}(\Omega)$  contains an isometric copy of $L^1[0,1]$.

\medskip

Finally, let us prove that $4)\Rightarrow 1)$:  Set $F:=F_1\cup F_\infty$ where by  $F_1$, $F_\infty$ we  denote the set of  atoms in $p^{-1}(\{1\})$ and  $p^{-1}(\{ +\infty\})$ respectively. In view of  $4)$, the cardinal of $F_\infty$ is finite so $L^{p(\cdot)}(F)$ has the Schur property (since it is isomorphic to $\ell_1$).

Assume, by contrary, that $L^{p(\cdot)}(\Omega)$ fails to have weakly normal structure. Standard arguments imply that  $L^{p(\cdot)}(\Omega)$ contains a weakly null diametral sequence (see for instance \cite[Lemma 4.1]{GK}), that is, a sequence $(y_n)$ with ${\rm diam}\,(\{y_n:n\in\mathbb{N}\})=1$ and such that
$\lim_n d(y_n,\textrm{co}\,\{y_1,...,y_{n-1}\})=1.$
In particular $
\lim_n\V y_n-y\V=1$ for all  $y\in \textrm{co}\,(\{y_n\}).$

We select  $u,v\in \textrm{co}\,(\{y_n\})$ that will be fixed in what follows.

 Note that we  can write $y_n=z_n+x_n$, where $z_n(t)=y_n(t)$ if $t\in F$ and zero otherwise, while $x_n=y_n-z_n$. It is clear that $(z_n)$ and $(x_n)$ are weakly null sequences, ${\rm supp}\,(z_n)\subset F$ and ${\rm supp }\,(x_n)\subset \Omega\setminus F$ for all $n\in\mathbb{N}$.
From the Schur property, the sequence $(z_n)$ 
is  norm convergent,  $\lim_n\V x_n-y_n\V=0$ and this implies that
$$
\lim_n\V x_n-y\V=1\quad  \forall y\in \textrm{co}\,(\{y_n\}).
$$
  The above condition implies that $$
\lim_n\V x_n-u\V=\lim_n\V x_n-v\V=\lim_n\left\V x_n-{u+v\over 2}\right\V=1.
$$
From the assumption $p_+(\Omega_f)<+\infty$ and Lemma \ref{properties}.c.i) we infer that
$$
\lim_n\rho( x_n-u)=\lim_n\rho( x_n-v)=\lim_n\rho\left( x_n-{u+v\over 2}\right)=1
$$
and consequently 
\begin{equation}\label{rho}
\lim_n\left[ \rho( x_n-u)+\rho( x_n-v)-2\rho\left( x_n-{u+v\over 2}\right)\right]=0.
\end{equation}

Define $u_F(t)=u(t)$ if $t\in F$, zero otherwise and $u_0=u-u_F$. Analogously we define $v_F$ and $v_0$. Thus $u=u_0+u_F$, $v=v_0+v_F$, where $u_0,v_0\in L^{p(\cdot)}(\Omega\setminus F)$ and $u_F,v_F \in L^{p(\cdot)}(F)$.
If we denote by $\rho_0(g):=\int_{\Omega\setminus F}|g|^{p(t)}d\mu$ and $\rho_F(g):=\rho(g)-\rho_0(g)$ for $g\in L^{p(\cdot)}(\Omega)$,  we have
$$
\rho(x_n-u)=\rho_0(x_n-u_0)+\rho_F(u_F)
$$
and a similar decomposition is obtained for $\rho(x_n-v)$ and $\rho\left(x_n-{u+v\over 2}\right)$.

Condition (\ref{rho}) is now translated to $A_1+A_2=0$, where
$$
A_1:=\rho_F( u_F)+\rho_F(v_F) -2\rho_F\left( {u_F+v_F\over 2}\right)
$$
and
$$
A_2:=\lim_n\left[ \rho_0( x_n-u_0)+\rho_0( x_n-v_0)-2\rho_0\left( x_n-{u_0+v_0\over 2}\right)\right].
$$
By convexity both $A_1,A_2\ge 0$, so we have that $A_1=A_2=0$.
Conse\-quently

{\footnotesize\begin{equation}\label{lulu}
\lim_n \int_{\Omega\setminus F} \left(|x_n(t)-u_0(t)|^{p(t)}+|x_n(t)-v_0(t)|^{p(t)}-2\left| x_n(t)-{u_0(t)+v_0(t)\over 2}\right|^{p(t)}\right)d\mu=0.
\end{equation}}

Due to the fact that  $p_+(\Omega_f)<+\infty$ and the remaining conditions in $4)$,  we have   $\sup_n\rho_0(x_n)<+\infty$. Furthermore, $1<p(t)<+\infty$ a.e. in $\Omega\setminus F$.  Consequently, for the vectors $u,v\in  \textrm{co}\,(\{y_n\})$ chosen beforehand, we know  of the existence of a  $\rho$-bounded sequence $(x_n)$ such that (\ref{lulu}) holds.

Applying  Lemma \ref{strictlyconvex}  for the set $\Omega\setminus F$, we deduce  that $u_0=v_0$ e.c.t. $\Omega\setminus F$ and 
  $u(t)=v(t)$ a.e. in $\Omega\setminus F$.
  
  Due to the arbitrariness of the vectors $u,v\in \textrm{co}\,(\{y_n\})$, we can deduce that for all $n,m$ we have that $y_n=y_m$ a.e. in $\Omega\setminus F$.
Since $(y_n)$ is a weakly null sequence, $y_n(t)=0$ a.e. in $\Omega\setminus F$ and $y_n\in L^{p(\cdot)}(F)$, which has the Schur property. Consequently $\lim_n\V y_n\V=0$ in contradiction with the fact that $(y_n)$ is diametral.

\end{proof}

At this stage we would like to highlight that the absence of an isometric copy of $L^1[0,1]$ is a necessary condition for having the $w$-FPP. 
As it was proved in Theorem \ref{char},  it turns out to be  an equivalence for the family of  VLSs.
Furthermore, according to Theorem \ref{char}, we can find plenty of examples  of nonreflexive VLSs that still have the  $w$-FPP. For instance, the Banach function space $L^{1+x}([0,1])$ is one of these spaces. Notice that this is not possible   for classic Lebesgue spaces, where $L^p(\Omega)$  has the $w$-FPP if and only if it is reflexive. In the particular case of a purely atomic measure,  sufficient conditions implying the $w$-FPP  had  been  studied previously in   \cite{CH, iran}  (see also \cite[Chapter 12]{Kirk-Sims} and references therein).

\section{Fixed point property and reflexivity: Maurey's result  and its converse in Variable Lebesgue Spaces}

As it was mentioned in the Introduction, one of the most relevant results  backing the 
 long-standing conjecture \lq\lq reflexivity implies   FPP"  was published by B. Maurey  in  1980 \cite{maurey}: {\it 
 every closed reflexive subspace of $L^1[0,1]$ has the FPP}. This together with P. Dowling and C. Lennard's converse    \cite{DL} lead to the following characterization:

\begin{theorem}\cite{maurey}\cite{DL}\label{ref} Let $X$ be a closed subspace of $L^1[0,1]$. Then $X$ is reflexive if and only if it  has the FPP.
\end{theorem}

Since every  $\sigma$-finite $L^1$-space is isometric to a probability space, Theorem \ref{ref} easily applies to the $\sigma$-finite case $L^1(\mu)$. The natural question that  rises straight away is whether one or the two implications in Theorem \ref{ref}
  may still hold in our variable setting.
In this section we will prove that  Maurey's result extends  to (nonreflexive) VLSs, whereas P. Dowling and C. Lennard's converse is no longer true.

\medskip

\begin{theorem}\label{reflexive} Let $ (\Omega,\Sigma,\mu) $ be a  $ \sigma $-finite  measure space  and $p_+<+\infty$. Let $X$ be a reflexive  subspace of  $L^{p(\cdot)}(\Omega)$. Then, $X$ satisfies the FPP.\end{theorem}
\begin{proof}
Let $X$ be a reflexive  subspace of  $L^{p(\cdot)}(\Omega)$ with $p_+<\infty$. 
If $L^{p(\cdot)}(\Omega)$ has the $w$-FPP,  we are done. Thus according to Theorem \ref{char}.4)   we can assume that $\mu(p^{-1}(\{1\}))>0$.

Denote $\Omega_1:=p^{-1}(\{1\})$, $\Omega_2:=\Omega\setminus \Omega_1$.
Define $Y:=\{f\in X: f\chi_{\Omega_2}=0\ a.e.\}$. Note that $Y$ is a closed subspace (possible empty) of $X$ and therefore $Y$ is reflexive. Furthermore,  $Y$ is embedded isometrically in $L^1(\Omega_1)$.

By contradiction, let us assume  that $X$ fails to have the FPP. Standard arguments show that there exist a convex weakly compact subset $K\subset X$ and $T:K\to K$ nonexpansive without fixed points. Furthermore, we can assume that $K$ is minimal $T$-invariant, $0\in K$,  ${\rm diam}\,(K)=1$ and there exists a weakly-null sequence $\{x_n\}$ in $K$ which is an approximate fixed point sequence.  As a consequence of Goebel-Karlovitz Lemma \cite[page 124]{GK}, $\lim_n \V x_n-x\V=1$ for all $x\in K$. Fix some $u, v\in K$. Proceeding as in the proof of  Theorem \ref{char}, we can assume that for $i=1,2$:
{\small
$$
\lim_n \int_{\Omega_i} \left(|x_n(t)-u(t)|^{p(t)}+|x_n(t)-v(t)|^{p(t)}-2\left| x_n(t)-{u(t)+v(t)\over 2}\right|^{p(t)}\right)d\mu=0.
$$
}
In particular, applying Lemma \ref{strictlyconvex} to $\Omega_2$, where $1<p(t)<\infty$ a.e.,  we deduce that $u\chi_{\Omega_2}=v\chi_{\Omega_2}$ a.e. Since $0\in K$, $u\chi_{\Omega_2}=0$ a.e. for all $u\in K$. Thus, $K$ is a convex weakly compact of  $Y$,  which has the FPP according to Maurey \cite{maurey}. This implies that $K$ is a singleton (since it is minimal) in contradiction to the the fact that diam$(K)=1$.

\end{proof}

The second part of this section is dedicated to study the 
 converse of Maurey's result.   Surprisingly, we  are going to prove that, under certain conditions over the function $p(\cdot)$, every nonreflexive $L^{p(\cdot)}(\Omega)$ contains a further nonreflexive Banach space fulfilling the FPP, in sharp contrast with the  $L^1[0,1]$-case and bringing to light   new intrinsic  features of variable Lebesgue spaces   that are not shared by their classic counterparts.

\medskip

In order to do that, we  introduce the concept of near-infinity concentrated norm  defined in \cite{jfa} for Banach spaces with a Schauder basis. Recall that if $\{e_n\}$ is a Schauder basis for a Banach space $X$, we denote by ${\rm supp}\,(x)=\{n\in\mathbb{N}: x(n)\ne 0\}$, $Q_k(x)=\sum_{n=k}^\infty x(n)e_n$ and $P_k(x)=\sum_{n=1}^{k-1}x(n)e_n$, where $x=\sum_{n=1}^\infty x(n) e_n\in X$. The norm  is said to be premonotone  for the basis $\{e_n\}$ when $\V Q_k\V\le 1$ for every $k\in\mathbb{N}$. For $k\in\mathbb{N}$ and $x\in X$, we say that $k\le x$ if $k\le \min\{{\rm supp}\,(x)\}$.

\begin{definition}\cite{jfa}\label{jfa-nic} Let $(X,\V\cdot\V)$ be a Banach space with a Schauder basis $\{e_n\}$. The norm is said to be  near-infinity concentrated (n.i.c.) if it has the following properties:

\begin{itemize}
\item[ (1)] The norm  is  sequentially separating  \cite{ja}, that is, for every $\epsilon>0$ there exists some $k\in\mathbb{N}$ such that
$$
\V x\V+\limsup_n\V x_n\V\le (1+\epsilon)\limsup_n\V x+x_n\V
$$
whenever $k\le x$ and $(x_n)$ is a block basic sequence of $\{e_n\}$.

\item[ (2)] The norm is premonotone.
\item[ (3)] There exist some  $R_0>5$ and $M\in [0,1)$ such that for every  $k\in \mathbb{N}$, we can find  a function $F_k:(0, \infty) \to [0, \infty)$ satisfying the following conditions:
\begin{itemize}
\item[(3.a)]  $\lim_{\lambda\to 0^+} {F_k(\lambda)\over \lambda}\le {M\over R_0}$.
\item[(3.b)]  $\forall$ $z\in X$  with $\V z\V\le R_0$,  ${\rm supp}\,(z)\subset[1,k]$  and for every bounded coordinate-null sequence $(x_n)\subset X$ with $\liminf_n \Vert x_n\Vert \geq 1$ we have:
$$ \limsup_n \Vert x_n+\lambda z\Vert \leq \limsup_n \Vert x_n\Vert +F_k(\lambda )\Vert z\Vert\qquad \forall \lambda\in (0,\infty).
$$
\end{itemize}
\end{itemize}

\end{definition}

The main result in \cite{jfa} is the following:

\begin{theorem}
Let $(X,\V\cdot\V)$ be a Banach space with a boundedly complete Schauder basis. If the norm $\V\cdot\V$ is  n.i.c., then $(X, \V\cdot\V)$ has the FPP.
\end{theorem}
 We will   make use of the following technical lemma:

\begin{lem}\cite[Lemma 3.3]{ja} \label{ss}Let $(X,\V\cdot\V)$ be a Banach space with a Schauder basis $\{e_n\}$. The norm $\V\cdot\V$ is sequentially separating if and only if
 $$
\lim_k \inf\{ \limsup_n\V x+x_n\V \}= 2,
$$
where the infimum is taken over all $x\in X$ such that $x=\sum_{i\ge k} x(i) e_i$
  with $\V x\V=1$  and all normalized block basic sequences $(x_n)\subset X$.

\end{lem}

We recall that a Banach space $X$ is said to be hereditarily $\ell_1$ if every closed subspace contains a further subspace which is isomorphic to $\ell_1$. Now the main result is the following:

\begin{theorem} \label{nic} Let $(\Omega,\Sigma,\mu)$ be a $\sigma$-finite measure space,  $p:\Omega\to [1,+\infty]$ be a measurable function. If $p_-\left(\Omega\setminus p^{-1}(\{1\})\right)=1$, then $L^{p(\cdot)}(\Omega)$ contains a closed subspace with the FPP which is hereditarily $\ell_1$, and therefore nonreflexive.

\end{theorem}

\begin{proof}
From the assumptions,  we can find a decreasing sequence $(\gamma_n)$ in $(1,+\infty)$ such that $\lim_n \gamma_n=1$ and $A_n:=\{t\in \Omega: \gamma_{n}< p(t)\le\gamma_{n-1}\}$ has positive measure. Let $f_n$ be  a normalized function in  $L^{p(\cdot)}(\Omega)$ with supp$(f_n)\subset A_n$ for every $n\in\mathbb{N}$. We next prove that    the closed subspace $X$ generated by the sequence $\{f_n\}$ is nonreflexive and satisfies the FPP:

Firstly, note that the sequence $\{f_n\}$ is a Schauder basis for $X$ which is boundedly complete (see for instance \cite[Definition 3.2.8]{A}), since the function $p(\cdot)$ is bounded in the union of all subsets $A_n$. Furthermore, the norm is premonotone.  We now check 
 conditions 1) and 3) in Definition \ref{jfa-nic}.

\medskip

 In order to prove $(1)$ we use Lemma \ref{ss}: Fix $k\in\mathbb{N}$ and let $x=\sum_{i\ge k}x(i)f_i$ with $\V x\V=1$. Let $(x_n)$ be a block basic sequence in $X$. Without loss of generality we can assume that
$\max$ supp $(x)<\min $ supp  $(x_n)$ in regards to their coordinates respect to the basis $(f_n)$. Note that if $t\in {\rm supp}\,(x)\cup {\rm supp}\,(x_n)$ then $p(t)<\gamma_{k-1}$. Besides,  $\rho(x)=\rho(x_n)=1$ for all $n\in\mathbb{N}$ since they are normalized vectors (see Lemma \ref{properties}.c.i  applied to a constant sequence and  taking in mind that the exponent function $p(\cdot)$ is bounded from above in the union of all the subsets $A_n$). For $r:=2^{1\over \gamma_{k-1}}$ we have:
{\small
$$
\begin{array}{lll}
\displaystyle{\rho\left(x+x_n\over r\right)}&= &\displaystyle{\rho\left({x\over r}\right)+\rho\left({x_n\over r}\right)}\\
& = &\displaystyle{ \int_{{\rm supp}\,(x)} \left({1\over r}\right)^{p(t)} |x(t)|^{p(t)} d\mu 
       +  \int_{{\rm supp}\, (x_n)} \left({1\over r}\right)^{p(t)} |x_n(t)|^{p(t)} d\mu}\\
     &  \ge&  \displaystyle{ \left({1\over r}\right)^{\gamma_{k-1}}\rho(x) +\left({1\over r}\right)^{\gamma_{k-1}}\rho(x_n)=1.}\\
\end{array}
$$
}
Therefore, $\limsup_n\V x+x_n\V\ge 2^{1\over \gamma_{k-1}}$. Taking limits when $k$ goes to infi\-nity we deduce  that the norm is sequentially separating. This in particular implies that $X$ is nonreflexive, since  it is hereditarily $\ell_1$ \cite[Corollary 7.7]{ja}.

\medskip

We next prove condition (3) in Definition \ref{jfa-nic}: Take any $R_0>0$ and $k\in\mathbb{N}$.
Let $(x_n)$ be a block basic sequence with $\liminf_n\V x_n\V\ge 1$ and $z\in X$ with $z=\sum_{i=1}^k z(i) f_i$ and $\V z\V\le R_0$. We can assume, without loss of generality,  that $x_n=\sum_{i=k+1}^\infty x_n(i)f_i$ and $\V x_n\V\ge 1$ for all $n\in\mathbb{N}$.
We start by proving:
\begin{equation}\label{a}
\V x_n+\lambda z\V\le \V x_n\V + \lambda^{\gamma_k}\V z\V^{\gamma_k} \qquad \forall \lambda\le R_0^{-1}.
\end{equation}

Indeed, note that $p(t)> \gamma_k$ for all $t\in {\rm supp}\,(z)$ and  $\lambda \V z\V\le 1\le \V x_n\V+\lambda^{\gamma_k}\V z\V^{\gamma_k}$ when $\lambda\le R_0^{-1}$.  Furthermore, $(\V x_n\V+\lambda^{\gamma_k}\V z\V^{\gamma_k})^{\gamma_k-1}\ge 1$.  This implies that:

$$
\begin{array}{lll}
&&\rho\left({x_n+\lambda z\over \V x_n\V+\lambda^{\gamma_k} \V z\V^{\gamma_k}}\right)=  \rho\left({ \V x_n\V\over \V x_n\V+\lambda^{\gamma_k} \V z\V^{\gamma_k}}{x_n\over \V x_n\V}\right)+ \rho\left( {\lambda \V z\V\over \V x_n\V+\lambda^{\gamma_k} \V z\V^{\gamma_k}}{z\over\V z\V}\right)\\
& \le &  { \V x_n\V\over \V x_n\V+\lambda^{\gamma_k} \V z\V^{\gamma_k}} \rho\left(x_n\over \V x_n\V\right)+\left( {\lambda \V z\V\over \V x_n\V+\lambda^{\gamma_k} \V z\V^{\gamma_k}}\right)^{\gamma_k}\rho\left({z\over \V z\V}\right)\\
& \le  &  { \V x_n\V\over \V x_n\V+\lambda^{\gamma_k} \V z\V^{\gamma_k}} + \left({\lambda \V z\V\over \V x_n\V+\lambda^{\gamma_k} \V z\V^{\gamma_k}}\right)^{\gamma_k}\\
  &= &\displaystyle{   {\V x_n\V(\V x_n\V+\lambda^{\gamma_k}\V z\V^{\gamma_k})^{\gamma_k-1} +\lambda^{\gamma_k} \V z\V ^{\gamma_k}\over (\V x_n\V+\lambda^{\gamma_k} \V z\V^{\gamma_k})^{\gamma_k}}}\\
  &\le &\displaystyle{   {\V x_n\V(\V x_n\V+\lambda^{\gamma_k}\V z\V^{\gamma_k})^{\gamma_k-1} +\lambda^{\gamma_k} \V z\V^{\gamma_k}(\V x_n\V+\lambda^{\gamma_k}\V z\V^{\gamma_k})^{\gamma_k-1}\over (\V x_n\V+\lambda^{\gamma_k} \V z\V^{\gamma_k})^{\gamma_k}}}=1.\\
\end{array}
$$

Therefore,  if $\lambda\le R_0^{-1}$:
$$
\begin{array}{lll}
\limsup_n\V x_n+\lambda z\V& \le & \limsup_n\V x_n\V + \lambda^{\gamma_k} \V z\V^{\gamma_k}\\
& = &\limsup_n\V x_n\V + \lambda^{\gamma_k} \V z\V^{{\gamma_k}-1}\V z\V\\
 &\le &\limsup_n\V x_n\V + \lambda^{\gamma_k} R_0^{\gamma_k-1}\V z\V.\\
\end{array}
$$
Thus,  we can consider $F_k(\lambda)=\lambda$ if $\lambda >R_0^{-1}$ and $F_k(\lambda)=\lambda^{\gamma_k}R_0^{\gamma_k-1}$ otherwise. Now it is clear that $\lim_{\lambda\to 0}{F_k(\lambda)\over \lambda}=0$ and this shows that the norm on $X$ is  n.i.c. Consequently  $X$ verifies the FPP as we wanted to prove.

\end{proof}

Likewise, it can proved that  the Luxemburg norm in Musielak-Orlicz  sequence spaces $\ell^{p_n}$  is near-infinity concentrated when the sequence $\{p_n\}\subset (1,+\infty)$ and $\lim_n p_n=1$. This drives us  to discover a family of  classic sequence Banach spaces which are nonreflexive and enjoy the FPP  without enduring any renorming process (and non-isomomorphic to $\ell_1$ \cite{Na}):

\begin{cor} \label{fpp2} Let $ \{p_n\}$ be a  sequence in $(1,+\infty)$  with $\lim_n p_n=1$. Then  the sequence  Musielak-Orlicz space $\ell^{p_n}$ endowed with the Luxemburg norm
enjoys the FPP.
\end{cor}

\begin{proof}
Indeed, consider the standard   Schauder basis $\{e_n\}$ in $\ell^{p_n}$. Take
 $A_n=\{n\}$ for every $n\in\mathbb{N}$,  the variable function  $p(n)=p_n$ with $\lim_n p_n=1$. Now we are done just mimicking the same steps as in the proof of Theorem \ref{nic}.

\end{proof}

Note that, as far as the authors are concerned, the  Musielak-Orlicz spaces $\ell^{p_n}$ with the Luxemburg norm are the first known nonreflexive Banach spaces satisfying the FPP  without having to undergo any  renorming process. 
We would like to point out that,  directly  referring  to  \cite{Z}, in \cite{ja} it was quoted that: \lq\lq $\ell^{p_n}$ with $\lim_n p_n=1$   endowed with the Luxemburg norm fails the FPP because it contains an asymptotically isometric copy of $\ell_1$" (see definition and its consequences in the next section).  Corollary \ref{fpp2}  shows  that this is not possible. In fact, after a careful reading, the authors of this manuscript strongly believe that there is a misunderstanding between the Luxemburg norm and the Orlicz norm  in \cite{Z}.

\medskip

Finally, we conclude this section with the following corollary:

\begin{cor} \label{subspaces} Let $(\Omega,\Sigma,\mu)$ be a $\sigma$-finite measure space and $p:\Omega\to [1,+\infty]$ be a measurable function such that $L^{p(\cdot)}(\Omega)$ is nonreflexive. 
Assume that one of the following conditions holds:

\begin{itemize}

\item[a)] $L^{p(\cdot)}(\Omega)$ contains an isometric copy of $\ell_\infty$,

\item[b)] $p^{-1}(\{ 1\})$ is essentially formed  by finitely many atoms at most, or

\item[c)]  $L^{p(\cdot)}(\Omega)$ does not contain isometrically $\ell_1$.

\end{itemize}

Then $L^{p(\cdot)}(\Omega)$ contains a further nonreflexive closed subspace  with the FPP.

\end{cor}

\begin{proof} If $L^{p(\cdot)}(\Omega)$ contains isometrically $\ell_\infty$,    in particular  it contains isometrically the Musielak-Orlicz space of  Corollary \ref{fpp2}. If $L^{p(\cdot)}(\Omega)$ is nonreflexive and  b)  or c) holds, using Theorem \ref{refle}.iii), either we can find an isometric copy of $\ell_\infty$ or
 $p_-\left(\Omega\setminus p^{-1}(\{1\})\right)=1$ and we can apply  Theorem \ref{nic}.

\end{proof}

\section{The failure of asymptotically isometric copies of $\ell_1$ in Variable Lebesgue Spaces}

The failure of the FPP is strongly connected to the existence of asympto\-tically isometric copies of $\ell_1$, notion that  was first defined by Hagler \cite{H1} and revitalized several years later by P. Dowling and C. Lennard  \cite{DL}.
 We recall that a Banach space $X$ is said to contain an asymptotically isometric copy  (a.i.c.) of $\ell_1$  if there exist a sequence $(z_n)\subset X$
 and some sequence $(\epsilon_n)\subset (0,1)$ with $\lim_n\epsilon_n=0$ such that
\begin{equation}\label{a}
\sum_{n=1}^\infty (1-\epsilon_n)|t_n|\le \left\V \sum_{n=1}^\infty t_n z_n\right\V\le \sum_{n=1}^\infty |t_n|
\end{equation}
for all $(t_n)\in\ell_1$. It was proved in \cite{DL} (see also \cite{handbook}) that: 

\begin{itemize}

\item[$1)$]
 {\it  If $(z_n)$ spans an a.i.c. of $\ell_1$, then the closed span $[z_n]$ fails the FPP. Consequently, every Banach space containing an asymptotically isometric copy of $\ell_1$ fails to have the FPP}.

\item[$2)$] {\it Every nonreflexive closed subspace of $L^1[0,1]$ contains an asympto\-tically isometric copy of $\ell_1$.}

\end{itemize}

Theorem \ref{nic} shows that  the verbatim translation of the assertion 2) above does not follow when $L^1[0, 1]$ is replaced by a nonreflexive VLS. From Corollary \ref{subspaces}  we  know that every nonreflexive  $L^{p(\cdot)}(\Omega)$ contains a further nonreflexive  subspace with the FPP whenever it does not contain $\ell_1$ isometrically. In this latter case, we wonder whether the whole space $L^{p(\cdot)}(\Omega)$  could sa\-tisfy the FPP.  If an a.i.c. of $\ell_1$ were found  in $L^{p(\cdot)}(\Omega)$, the answer would be negative by assertion 1) above.  A natural question  arises:   Assume that $L^{p(\cdot)}(\Omega)$ does not contain $\ell_1$ isometrically. Can $L^{p(\cdot)}(\Omega)$ contain    an a.i.c. of $\ell_1$?
In this last section we  prove that the absence of an a.i.c. of $\ell_1$  is, in fact, the general rule in nonreflexive VLS.  To achieve our goals we need to introduce the subsequence splitting property  and some preliminar results
  (see for instance \cite{W}).

\begin{definition} A Banach function space $X$ is said to have  the subsequence splitting property (SSP) if for every bounded sequence $(f_n)\subset X$ there is a subsequence $(f_{n_k})$ and sequences $(g_k)$, $(h_k)$ with $|g_k|\wedge |h_k|=0$ and     $f_{n_k}=g_k+h_k$ for all $k\in\mathbb{N}$ such that
\begin{itemize}

\item[i)] The sequence $(g_k)$ is equi-integrable in $X$.
\item[ii)] The $h_k$'s are pairwise disjoint.

\end{itemize}
\end{definition}
 In the framework of a Banach lattice with an order continuous norm and  a weak unit, the SSP was completely characterized in \cite[Theorem 2.5]{W}:

\begin{theorem}\cite{W} \label{SSPt}Let $X$ be a Banach lattice with order continuous norm and  weak unit. Then $X$ has the SSP if and only if $\ell_\infty^n$'s are not equi-normably embedded into $X$ (see Definition 2.4 in \cite{W}).
\end{theorem}

As a consequence  we can deduce:

\begin{cor}\label{SSP}

Let $(\Omega,\Sigma,\mu)$ be a $\sigma$-finite measure space and $p:\Omega\to [1,+\infty]$ be a measurable function.   Then $L^{p(\cdot)}(\Omega)$ verifies the SSP  when $L^{p(\cdot)}(\Omega)$ contains no isometric copy of $\ell_\infty$.

\end{cor}

\begin{proof}
Since $(\Omega,\sigma,\mu)$ is $\sigma$-finite, it is not difficult to find  $g>0$ a.e. and therefore $g$ is a weak unit. The absence of isometric copies of $\ell_\infty$ implies that $P_+:=p_+(\Omega_f)<+\infty$, $p^{-1}(\{+\infty\})$ is essentially formed by  finitely many atoms at the most and the norm is order continuous. We next prove that $\ell_\infty^n$'s cannot be equi-normably embedded in $L^{p(\cdot)} (\Omega_f)$   (which automatically implies the same assertion for $L^{p(\cdot)} (\Omega)$ and the SPP by Theorem \ref{SSPt}):

Let $\epsilon>0$ and assume that for every $n\in\mathbb{N}$ we can find  $\{g_1,\cdots, g_n\}$ disjointly supported functions in $ L^{p(\cdot)} (\Omega_f)$ with $\V g_i\V=1$ for $1\le i\le n$ and $\left\V \sum_{i=1}^n g_i\right\V\le 1+\epsilon$. Using Lemma \ref{properties}.c.i) we know that $\rho(g_i)=1$  for $1\le i\le n$ and consequently:
$$
1\ge \rho\left({\sum_{1=1}^n g_i\over 1+\epsilon}\right)=\sum_{i=1}^n \rho\left({g_i\over 1+\epsilon}\right)\ge \sum_{i=1}^n \left({1\over 1+\epsilon}\right)^{P_+}\rho(g_i)={n\over (1+\epsilon)^{P_+}},
$$
which implies that $n\le (1+\epsilon)^{P_+}$ which is not possible. 

\end{proof}

Note that the following stability properties for asymptotically isometric copies of $\ell_1$ are easily obtained from the inequalities  in (\ref{a}):

\begin{itemize}

\item[i)] Every subsequence of an a.i.c. of $\ell_1$ and  every absolutely  convex block basic sequence of an a.i.c. of $\ell_1$ span again an a.i.c. of $\ell_1$.

\item[ii)]  If $(z_n)$ spans an a.i.c. of $\ell_1$,   $\V u_n\V\le 1$ and $\lim_n\V  z_n-u_n\V=0$,  removing finitely many  terms if necessary,  $(u_n)$ spans an a.i.c. of $\ell_1$.

\end{itemize}

\medskip

Finally, the main theorem of this section  aims  to show the lack of asymptotically isometric copies of $\ell_1$ in VLSs unless  the trivial case is considered:

\begin{theorem} \label{last} Let $(\Omega,\Sigma,\mu)$ be a $\sigma$-finite  measure space and $p:\Omega\to [1,+\infty]$ be a measurable function.The following are equivalent:

\begin{itemize}

\item[$a)$] $L^{p(\cdot)}(\Omega)$ contains an isometric copy of $\ell_1$.

\item[$b)$] $L^{p(\cdot)}(\Omega)$ contains an  asymptotically isometric copy of $\ell_1$.

\end{itemize}
\end{theorem}

\begin{proof} $a)$ implies $b)$ is straightforward. Let us prove
 $b)$ implies $a)$: Assume  
 $L^{p(\cdot)}(\Omega)$ has a sequence $(f_n)$ spanning an a.i.c. of $\ell_1$.
Denote, as usually,  $\Omega_f=\{t\in \Omega: p(t)<\infty\}$ and we set $p^{-1}(\{1,+\infty\})=A\cup B$ where $A$ is the atomic part and $B$  a nonatomic set.

Assume, by contradiction,  that $L^{p(\cdot)}(\Omega)$ contains no  isometric copy of $\ell_1$. This automatically  implies that  $p_+(\Omega_f)<+\infty$, $A$ is formed by finitely many atoms (possible empty) and $\mu(B)=0$.
In view of  these assumptions we can assume that the Luxemburg  norm is order continuous and $f_n=f_n\chi_{\Omega\setminus B}$ in $L^{p(\cdot)}(\Omega)$.
  We split the proof in several steps:

\medskip

 Step 1: We can assume that the a.i.c. $(f_n)$ is pairwise disjoint and ${\rm supp}\,(f_n)\subset \Omega\setminus p^{-1}\{1,+\infty\})$ for all $n\in\mathbb{N}$ (since $p^{-1}(\{1,+\infty\})=A\cup B$,  $A$ is a collection of  a finite number of atoms and $\mu(B)=0$):

\medskip

  From Corollary \ref{SSP}, $L^{p(\cdot)}(\Omega)$ has the SSP and without loss of generality we can assume that $f_n=g_n+h_n$
where $|g_n|\wedge |h_n|=0$,  the $h_n$'s are pairwise disjoint, $(g_n)$ is equi-integrable and therefore  relatively weakly compact (see \cite[Definition 3.6.1 and Proposition 3.6.5]{MN}).
 Taking  a further subsequence, we can assume that $\{g_n\}$ is weakly convergent and so the sequence $g_n^\prime={g_{2n}-g_{2n-1}\over 2}$ is weakly convergent to zero. Using Mazur's Theorem, there exist $p_1\le q_1<p_2\le q_2<\cdots $ and a nonnegative sequence $\{\lambda_i\}$ such that $\sum_{i=p_n}^{q_n}\lambda_i=1$ and the sequence $\{G_n:=\sum_{i=p_n}^{q_n} \lambda_i g_i^\prime\}$ is norm null-convergent.

Define $F_n:=\sum_{i=p_n}^{q_n} \lambda_i {f_{2i}-f_{2i-1}\over 2}$ and $H_n=\sum_{i=p_n}^{q_n} \lambda_i {h_{2i} - h_{2i-1}\over 2}$ so 
 $F_n=H_n+G_n$ and $\lim_n \V F_n-H_n\V=\lim_n\V G_n\V=0$. From i) and ii)  above, $\{F_n\}$ spans an a.i.c. of $\ell_1$ and so does  $\{H_n\}$, which is pairwise disjoint.

\medskip

Step 2: We can suppose   that $p_-( {\rm supp}\,(f_n))>1$ for all $n\in\mathbb{N}$: 

 Take  a nonincreasing sequence $(\gamma_k)_k\subset (1,+\infty)$ with $\lim_k\gamma_k=1$.
Fix some $n\in\mathbb{N}$.
 Since ${\rm supp}\,(f_n)\subset \Omega\setminus p^{-1}(\{1,+\infty\})$ for all $n\in\mathbb{N}$, the sequence
$$
f_n\chi_{p^{-1}([1,\gamma_k])}\to_k 0 \hbox{ a.e. }t\in \Omega
$$
and therefore $\lim_k \rho(f_n\chi_{p^{-1}([1,\gamma_k])})=0$.
Since $p_+(\Omega_f)<+\infty$, from Lemma \ref{properties}.c.ii) we have $\lim_k\V f_n\chi_{p^{-1}([1,\gamma_k])}\V=0$.
Thus, for all $n\in\mathbb{N}$ we can consider some $k_n\in\mathbb{N}$ such that 
$$\lim_n \V  f_n - f_n\chi_{p^{-1}(\gamma_{k_n},\infty]}\V=0.
$$
Define $h_n:= f_n\chi_{p^{-1}(\gamma_{k_n},\infty]}$. Now the sequence $(h_n)$ is pairwise disjoint, spans an a.i.c. of $\ell_1$  and $p_-({\rm supp}\,(h_n))\ge \gamma_{k_n}>1$.

\medskip

Once that  Steps 1 and 2 have been proved, we can assume that $L^{p(\cdot)}(\Omega)$ contains a  pairwise disjoint sequence $(f_n)$ with $p_n:=p_-( A_n)>1$,  where $A_n:={\rm supp}\,(f_n)$ for all $n\in \mathbb{N}$  and spanning an a.i.c. of $\ell_1$.
 Let $(\epsilon_n)$ be the  null sequence  verifying  the inequalities (\ref{a}) for  $(f_n)$ and
take a subsequence $(\epsilon_{n_k})$ such that $\epsilon_{n_k}<{1\over k^2}$.   Defining $g_k=f_{n_k}$,  $\eta_k=\epsilon_{n_k}$,  the sequence $(g_k)$   spans an a.i.c. of $\ell_1$  and $\lim_k k \eta_k=0$.
Hence, we can assume that the sequences $(f_n)$ and $(\epsilon_n)$  obtained in Step 2 additionally satisfy  $\lim_n n\epsilon_n=0$.

\medskip

 In particular,
 since $p_1>1$, there exists some $n_0\in\mathbb{N}$ such that:
\begin{equation}\label{n}
n^{1-p_1\over 2}+ n \epsilon_n<1-\epsilon_1\quad\mbox{for all $n\ge n_0$}.
\end{equation}

Consider the function $h:[1,+\infty)\to [0,+\infty)$ by $h(t)=t^{(p_1-1)}(t-1)$, which  is strictly increasing in $[1,+\infty)$, 
 $h(1)=0$ and $\lim_{t\to+\infty}h(t)=+\infty$. Thus $h$ is bijective and bicontinuous.  We claim that
$$
\left\V{f_{1}\over n}+f_n\right\V\le h^{-1}\left({1\over n^{p_1}}\right) \mbox{for all $n\in\mathbb{N}$.}
$$
Indeed, fix some $n\in\mathbb{N}$. We rename   $s:=h^{-1}\left({1\over n^{p_1}}\right)>1$.
Then:
$$
\begin{array}{lll}
\displaystyle{\rho\left({{f_1\over n}+f_n\over s}\right)}&=&\displaystyle{\int_{A_1}\left(1\over ns\right)^{p(t)}|f_1(t)|^{p(t)} d\mu+\int_{A_n}\left(1\over s\right)^{p(t)}|f_n(t)|^{p(t)} d\mu}\\
&\le & \displaystyle{\left(1\over ns\right)^{p_1}\rho(f_1)+{1\over s}\rho(f_n)\le \left(1\over ns\right)^{p_1}+{1\over s}=1}\\
\end{array}
$$
and the claim is proved. 
From the left-hand side of the inequalities (\ref{a}) and using (\ref{n}), for all $n\ge n_0$ we have that:
$$
h^{-1}\left({1\over n^{p_1}}\right)\ge {1-\epsilon_1\over n}+1-\epsilon_n\ge n^{-1-{p_1}\over 2} + 1.
$$
Since the function $h$ is strictly increasing, the above implies that
$$
h\left(n^{-{{p_1}+1\over 2}}+1\right)=\left(n^{-{{p_1}+1\over 2}}+1\right)^{p_1-1}n^{-{p_1+1\over 2}}\le{1\over n^{p_1}}
$$
and
$$
\left(n^{-{p_1+1\over 2}}+1\right)^{p_1-1}\le n^{1-p_1\over 2}.
$$
Taking limits when $n$ goes to infinity, we arrive at the inequality $1\le 0$, which shows that $L^{p(\cdot)}(\Omega)$ cannot have an a.i.c. of $\ell_1$ in the absence of isometric copies of $\ell_1$ and the proof is complete.

\end{proof}

It is uncertain for the authors whether  every nonreflexive $L^{p(\cdot)}(\Omega)$ without an isometric copy of $\ell_1$ may satisfy  the FPP. What can at once be deduced from Theorem \ref{last}  is that, in order  to assert the failure of the FPP,  
either a precise counterexample would need to  be found or new alternative techniques would need to emerge, since asymptotically isometric copies of $\ell_1$ do not play any relevant role in contrast to the $L^1$-case.

A seemingly much easier problem does not have a precise answer yet:
 Consider a purely atomic $\sigma$-finite measure space and a bounded sequence $(p_n)\subset [1,+\infty)$ with $\liminf_n p_n=1$ (which implies that it is nonreflexive).  Does the Musielak-Orlicz space $\ell^{p_n}$ with the Luxemburg norm  have the FPP? (We only know that the answer is positive when $\lim_np_n=1$ due to Corollary \ref{fpp2}).

\end{document}